\nonstopmode
\documentclass[reqno,10pt]{amsart}
\usepackage{latexsym}
\usepackage{fancyhdr}
\usepackage{amsmath, amssymb}
\usepackage[ansinew]{inputenc}
\usepackage[all]{xy}
\usepackage{pdflscape}
\usepackage{longtable}
\usepackage{rotating}
\usepackage{verbatim}
\usepackage{hyperref}
\usepackage{subfigure}
\usepackage{mathrsfs}
\usepackage{cleveref}
\usepackage{mdwlist}
\usepackage{color}

\newcounter{mnotecount}[section]

\newcommand{\rmnote}[1]{}

\DeclareFontFamily{U}{mathb}{\hyphenchar\font45}
\DeclareFontShape{U}{mathb}{m}{n}{
      <5> <6> <7> <8> <9> <10> gen * mathb
      <10.95> mathb10 <12> <14.4> <17.28> <20.74> <24.88> mathb12
      }{}
\DeclareSymbolFont{mathb}{U}{mathb}{m}{n}
\DeclareFontSubstitution{U}{mathb}{m}{n}

\let\dot\relax
\DeclareMathAccent{\dot}{0}{mathb}{"39}
\let\ddot\relax
\DeclareMathAccent{\ddot}{0}{mathb}{"3A}
\let\dddot\relax
\DeclareMathAccent{\dddot}{0}{mathb}{"3B}
\let\ddddot\relax
\DeclareMathAccent{\ddddot}{0}{mathb}{"3C}

\theoremstyle{plain}
\newtheorem*{theorem*}{Theorem}
\newtheorem{theorem}{Theorem}
\newtheorem*{lemma*}{Lemma}
\newtheorem{lemma}[theorem]{Lemma}
\newtheorem*{proposition*}{Proposition}
\newtheorem{proposition}[theorem]{Proposition}
\newtheorem*{corollary*}{Corollary}
\newtheorem{corollary}[theorem]{Corollary}
\newtheorem*{claim*}{Claim}

\newtheorem*{conjecture*}{Conjecture}

\newtheorem*{question*}{Question}
\theoremstyle{definition}
\newtheorem*{definition*}{Definition}

\newtheorem*{example*}{Example}

\newtheorem*{algorithm*}{Algorithm}
\newtheorem*{remark*}{Remark}
\newtheorem*{remarks*}{Remarks}
\newtheorem{remark}[theorem]{Remark}

\newtheorem*{convention*}{Convention}

\newcommand{\al}{\alpha}
\newcommand{\be}{\beta}

\newcommand{\ep}{\epsilon}

\newcommand{\rh}{\rho}

\newcommand{\si}{\sigma}

\newcommand{\vh}{\varphi}
\newcommand{\ch}{\chi}

\newcommand{\om}{\omega}

\newcommand{\De}{\Delta}

\newcommand{\Si}{\Sigma}

\newcommand{\N}{\mathbb{N}}

\newcommand{\R}{\mathbb{R}}

\newcommand{\cC}{\mathcal{C}}

\newcommand{\cE}{\mathcal{E}}

\newcommand{\cG}{\mathcal{G}}

\newcommand{\cS}{\mathcal{S}}

\newcommand{\fM}{\mathfrak{M}}

\newcommand{\fW}{\mathfrak{W}}

\newcommand{\fa}{\mathfrak{a}}

\newcommand{\fg}{\mathfrak{g}}

\newcommand{\fk}{\mathfrak{k}}

\newcommand{\fp}{\mathfrak{p}}

\newcommand{\p}{\partial}

\newcommand{\on}{\operatorname}

\newcommand{\sr}[1]%
{\ifmmode{}^\dagger\else${}^\dagger$\fi\ifvmode
\vbox to 0pt{\vss
 \hbox to 0pt{\hskip\hsize\hskip1em
 \vbox{\hsize3cm\raggedright\pretolerance10000
 \noindent #1\hfill}\hss}\vss}\else
 \vadjust{\vbox to0pt{\vss%
 \hbox to 0pt{\hskip\hsize\hskip1em%
 \vbox{\hsize3cm\raggedright\pretolerance10000%
 \noindent #1\hfill}\hss}\vss}}\fi%
}

\newcommand{\A}{\;\forall}
\newcommand{\E}{\;\exists}

\newcommand{\grad}{\on{grad}}
\newcommand{\Tr}{\on{Tr}}

\newcommand{\ol}{\overline}

\title[]
{Ultradifferentiable Chevalley theorems \\ and isotropic functions}

\author[A.~Rainer]{Armin Rainer}

\address{Fakult\"at f\"ur Mathematik, Universit\"at Wien, 
Oskar-Morgenstern-Platz~1, A-1090 Wien, Austria 
\& University of Education Lower Austria,
Campus Baden M\"uhlgasse 67, A-2500 Baden}
\email{armin.rainer@univie.ac.at}

\begin{document}

\begin{abstract}
	We prove ultradifferentiable Chevelley restriction theorems for a wide range of ultradifferentiable classes. 
	As a special case we find that isotropic functions, i.e., functions defined on the vector space of real symmetric matrices invariant 
	under the action of the special orthogonal group by conjugation, possess some ultradifferentiable regularity if 
	and only if their restriction to diagonal matrices has the same regularity. 
\end{abstract}

\thanks{The author was supported by the Austrian Science Fund (FWF), START Programme Y963.}
\date{\today}

\maketitle

\section{Introduction}

Let the special orthogonal group $\on{SO}(n)$ act by conjugation on the vector space $\on{Sym}(n)$ of real symmetric $n \times n$ matrices. 
Functions $f : \on{Sym}(n) \to \R$ that are invariant under this action are called \emph{isotropic}, i.e.\
\[
	f(S A S^t) = f(A) \quad \text{ for all } A \in \on{Sym}(n), ~ S \in \on{SO}(n).
\]
By the spectral theorem, every $\on{SO}(n)$-orbit $\on{SO}(n) \cdot  A$ intersects the subspace $\on{Diag}(n) \cong \R^n$ of diagonal matrices 
orthogonally with respect to the invariant inner product $\langle A,B \rangle = \on{Tr}(AB^t) = \on{Tr}(AB)$. 
The intersection is the orbit of the symmetric group $\on{S}_n$ which acts by permuting the eigenvalues of $A$.
Then $f : \on{Sym}(n) \to \R$ is isotropic if and only if 
\[
	f(A) = F(a_1,\ldots,a_n), \quad A \in \on{Sym}(n),  
\]
for some unique symmetric function $F : \R^n \to \R$, 
where $a_1,\ldots,a_n$ are the eigenvalues of $A$ repeated according to their multiplicity. 
Isotropic functions are important in continuum mechanics, in particular, elasticity. 

It is well-known that the map that assigns to a real symmetric matrix $A$ its $n$-tuple of eigenvalues 
(e.g.\ in decreasing order) is Lipschitz (even difference-convex) but not $C^1$. 
Thus it is surprising that $f$ is smooth if and only if $F$ is smooth. 
This follows from a result of Glaeser \cite{Glaeser63F} (see also Schwarz \cite{Schwarz75} and Mather \cite{Mather77} for generalizations): 
every symmetric $F \in C^\infty(\R^n)$ can be written in the form
\[
	F(x) = G(\si_1(x),\ldots,\si_n(x)),
\]
where $G \in C^\infty(\R^n)$ and $\si_i$ are the elementary symmetric functions in $n$ variables. 
Consequently, $$f(A) = G((-1)^1\on{Tr} (\wedge^1 A) , \ldots, (-1)^n\on{Tr}( \wedge^n A))$$ is the composite of two $C^\infty$-maps. 
Glaeser's theorem for finite differentiability involves an unavoidable loss of regularity and can hence not be 
used to determine the regularity of $f$. Indeed, for a symmetric $C^{n r}$-function $F$ the function $G$ is in general only of class $C^r$; 
see Barban{\c{c}}on \cite{Barbancon72} and Rumberger \cite{Rumberger98}. 
Nevertheless it is true that, for any $r \in \N$, an isotropic function $f(A) = F(a_1,\ldots,a_n)$ is $C^r$ if and only if $F$ is $C^r$. 
This was proved by Ball \cite{Ball84} for $r=0,1,2,\infty$ and by Sylvester \cite{Sylvester:1985aa} for all $r = 0,1,2,\cdots,\infty$. 
Later \v{S}ilhav\'y \cite{Silhavy2000} gave a simple elementary proof of Sylvester's result and an inductive formula for the derivatives of $f$; 
see also Scheuer \cite{Scheuer:2018aa}.
The simpler H\"older case $C^{r,\al}$ for $0 < \al <1$ is already contained in \cite{Ball84}. 
For applications in elasticity see the discussion of the stored energy function in \cite[Section 6]{Ball84}.

In this paper we will show that the same phenomenon ``permanence of regularity'' between an isotropic function $f$ and its 
symmetric companion  $F$ holds for \emph{ultra\-differentiable functions}. 
These are $C^\infty$-function with certain growth restrictions for their iterated derivatives
which define the ultradifferentiable class. For instance, the real analytic class is defined by the Cauchy estimates. 
Modification of the Cauchy estimates in terms of a weight sequence gives rise to the classical Denjoy--Carleman classes (among 
them the Gevrey classes which are important in PDEs). 
Braun--Meise--Taylor classes arose from measuring the regularity in terms of prescribing the decay of the Fourier transform.

We will work in a very general framework for ultradifferentiable analysis which comprises all classically studied classes (notably, 
the aforementioned ones). 
Similarly as for finite differentiability the corresponding Glaeser theorem involves a strict loss of regularity and is hence not 
applicable; 
see Bronshtein\cite{Bronshtein86,Bronshtein87} and \cite{RainerDC}. 
It is worth mentioning that for Denjoy--Carleman classes we require 
that the weight sequence has moderate growth whereas in the case of Braun--Meise--Taylor classes our results apply 
for all standard weight functions. 
(Precise definitions are given in \Cref{sec:ultra}.)  
It does not matter for the problem whether the class is quasianalytic or not.
In particular, we get as a corollary that an isotropic function $f$ is real analytic if and only if	
its symmetric companion $F$ is real analytic. 
Furthermore we shall also deduce a version of the result for \emph{Gelfand--Shilov classes}, i.e., ultradifferentiable rapidly 
decreasing functions for which the defining bounds are global in contrast to the aforementioned regularity classes.

The results for isotropic functions (\Cref{isotropic}) will be special cases of the ultradifferentiable Chevalley restriction theorems
(\Cref{Chevalley} and \Cref{GelfandShilov}) 
that we shall prove in \Cref{sec:Chevalley}. 
The latter are formulated 
for Cartan decompositions of real semisimple Lie algebras of noncompact type; see the setting in \Cref{sec:Chevalley}.  
The proof follows closely the one given by Dadok \cite{Dadok:1982aa} for the $C^\infty$-case which is based on 
the analysis of the Laplace operator on invariant functions and a weak elliptic regularity result. 
We shall combine this analysis with lacunary regularity results for ultradifferentiable classes. 
These lacunary regularity results are reviewed and adapted to our general ultradifferentiable setting in \Cref{sec:lacunary}.   
For the Chevalley theorem in Gelfand--Shilov classes we will use their invariance under the Fourier transform.

\subsection*{Notation}
We will use multiindex notation with $D_j = - i \p_j$, where $i = \sqrt{-1}$, 
and $D^\al = D_1^{\al_1} D_2^{\al_2} \cdots D_n^{\al_n}$ such that the  
Fourier transform $$\widehat f(\xi) = \int f(x) e^{- i \langle x,\xi\rangle } \, dx$$ on the Schwartz class $\cS(\R^n)$ satisfies 
$\widehat {D_jf}(\xi) = \xi_j \widehat f(\xi)$ and $\widehat{x_j f(x)} = - D_j \widehat f$. 

\section{Ultradifferentiable functions} \label{sec:ultra}

\subsection{Denjoy--Carleman classes} \label{sec:DC}

Let $U \subseteq \R^n$ be open. 
Let $M=(M_k)$ be a positive sequence.
For $\rh>0$ and compact $K\subseteq U$ consider the seminorm
\[
	\|f\|^M_{K,\rh} := \sup_{\substack{x \in K\\\al\in \N^n}} \frac{|D^\al f(x)|}{\rh^{|\al|} M_{|\al|}}, \quad f \in \cC^\infty(U).  
\]
The \emph{Denjoy--Carleman class of Roumieu type} $\cE^{\{M\}}$ is defined by 
\[
	\cE^{\{M\}}(U) := \big\{f \in \cC^\infty(U) : \A  K \Subset U \E \rh >0 : \|f\|^M_{K,\rh} <\infty\big\}
\]
and the \emph{Denjoy--Carleman class of Beurling type}  $\cE^{(M)}$ by 
\[
	\cE^{(M)}(U) := \big\{f \in \cC^\infty(U) : \A  K \Subset U \A \rh >0 : \|f\|^M_{K,\rh} <\infty\big\},
\]
We endow these spaces with their natural locally convex topologies.
The study of Denjoy--Carleman classes started around 1900 with the work of E.\ Borel.

We shall assume that the sequence $M=(M_k)$ is
\begin{enumerate}
	\item logarithmically convex, i.e.\ $M_k^2 \le M_{k-1} M_{k+1}$ for all $k$, and satisfies
	\item $M_0 = 1 \le M_1$ and 
	\item $M_k^{1/k} \to \infty$.
\end{enumerate} 
In that case we say that $M$ is a \emph{weight sequence}. 
It is easy to see that for a weight sequence $M$ the sequence $M_k^{1/k}$ (and thus also $M$) 
is increasing and $M_{k}M_{\ell} \le M_{k+\ell}$ for all 
$k,\ell$.

Let $M=(M_k)$ and $N=(N_k)$ be positive sequences. Then boundedness of the sequence $(M_k/N_k)^{1/k}$ is a sufficient condition for the 
inclusions 
$\cE^{\{M\}} \subseteq \cE^{\{N\}}$ and $\cE^{(M)} \subseteq \cE^{(N)}$ (this means that the inclusions hold on all open sets). 
The condition is also necessary provided that $M$ satisfies (1), see \cite{Thilliez08} and \cite{Bruna80/81}. 
For instance, stability of the classes $\cE^{\{M\}}$ and $\cE^{(M)}$ by derivation is equivalent to boundedness of 
the sequence $(M_{k+1}/M_k)^{1/k}$ 
(for the necessity we assume that $M$ satisfies (1)). 
If $(M_k/N_k)^{1/k} \to 0$ then $\cE^{\{M\}} \subseteq \cE^{(N)}$, and conversely provided that $M$ satisfies (1). 
Hence sequences $M$ and $N$ satisfying (1) are called \emph{equivalent} if 
there is a constant $C>0$ such that $C^{-1} \le (M_k/N_k)^{1/k}\le C$; 
this is precisely the case if they defined the same Denjoy--Carleman classes.

Of particular importance in the theory of differential equations are the \emph{Gevrey classes} $\cG^s$, for $s \ge 1$. 
These are by definition 
the 
Roumieu type classes associated with the weight sequence $M_k= k!^s$, i.e., $\cG^s = \cE^{\{(k!^s)_k\}}$.  
For $s=1$ we get the class of real analytic functions $\cG^1 = \cE^{\{(k!)_k\}} = \cC^\om$ in the Roumieu 
case and the restrictions of entire functions $\cE^{((k!)_k)}$ in the Beurling case.
 
The inclusion $C^{\om} \subseteq \cE^{\{M\}}$ (as well as the inclusion $\cE^{((k!)_k)} \subseteq \cE^{(M)}$) is equivalent to the condition
\begin{enumerate}
	\item[(4)] $(k!/M_k)^{1/k}$ is bounded. 
\end{enumerate}
The inclusion $C^\om \subseteq \cE^{(M)}$ is equivalent to the stronger condition
\begin{enumerate}
	\item[($4'$)]  $(k!/M_k)^{1/k} \to 0$ as $k \to \infty$. 
\end{enumerate}

A positive sequence $M=(M_k)$ is said to have \emph{moderate growth} if 
\begin{enumerate}
	\item[(5)] $\exists C>0 \A k,\ell \in \N : M_{k+\ell} \le C^{k+\ell} M_k M_\ell$. 
\end{enumerate}
This condition entails that the corresponding classes $\cE^{\{M\}}$ and $\cE^{(M)}$ are stable by derivation.

Note that the Gevrey classes $\cG^s$, for $s\ge 1$, satisfy all of the conditions (1)--(5),
if $s>0$ also $(4')$ is satisfied.

\subsection{General ultradifferentiable classes}

By a \emph{weight matrix} we mean a   
family $\fM$ of weight sequences $M$ which is totally ordered with respect to the pointwise order relation on sequences. 
For a weight matrix $\fM$ and an open subset $U \subseteq \R^n$  
we consider the \emph{Roumieu class}
\begin{align*} 
\cE^{\{\fM\}}(U) &:= \big\{f \in \cC^\infty(U) : \A  K \Subset U \E M \in \fM \E \rh >0 : \|f\|^M_{K,\rh} <\infty\big\}
\end{align*}
and the \emph{Beurling class}
\begin{align*} 
	\cE^{(\fM)}(U) &:= \big\{f \in \cC^\infty(U) : \A  K \Subset U \A M \in \fM \A \rh >0 : \|f\|^M_{K,\rh} <\infty\big\}
\end{align*} 
with their natural locally convex topologies.
Clearly every Denjoy--Carleman class $\cE^{\{M\}}$ and $\cE^{(M)}$ is a ultradifferentiable class of this kind (then 
$\fM$ consists just of the weight sequence $M$). 

In the following we will consider weight matrices $\fM$ with additional properties which will depend 
on the type, i.e.\ Beurling or Roumieu, of the class:
\begin{enumerate}
	\item[($B$)] (Beurling case) For all $M \in \fM$ we have  
	\begin{gather} \tag{$B_1$}  
		(k!/ M_k)^{1/k} \to 0 \quad \text{ as } k \to \infty, 
		\\
		\tag{$B_2$} 
		\A M \in \fM \E L \in \fM \E C>0 \A k,\ell \in \N : L_{k+\ell} \le C^{k + \ell} M_k M_\ell.
	\end{gather}
	\item[($R$)] (Roumieu case) For all $M \in \fM$ we have that $(k!/ M_k)^{1/k}$ is bounded 
	and 
	\begin{gather} \tag{$R_1$} 
		(k!/ M_k)^{1/k} \quad \text{ is bounded, } 
		\\	\tag{$R_2$} 
		\A M \in \fM \E N \in \fM \E C>0 \A k,\ell \in \N : M_{k+\ell} \le C^{k + \ell} N_k N_\ell.
	\end{gather}
\end{enumerate} 
We will call a weight matrix $\fM$ [\emph{regular}] if it satisfies ($B$) in the Beurling case and ($R$) in the Roumieu case. 
Under these conditions the classes $\cE^{[\fM]}$ contain the class of real analytic functions and are stable under differentiation. 
We use square brackets $[~]$ in statements that hold in the Beurling case $(~)$ as well as in the Roumieu case $\{~\}$ 
under the respective assumptions.

For Denjoy--Carleman classes either of the conditions ($B_2$) and ($R_2$) reduces to the moderate growth condition
\ref{sec:DC}(5). 

All our results will apply for \emph{Braun--Meise--Taylor classes} $\cE^{[\om]}$. Here $\om$ 
is a \emph{weight function}, i.e., a continuous increasing functions $\om : [0,\infty) \to [0,\infty)$ with $\om(0) =0$ 
and 
$\lim_{t \to \infty} \om(t) = \infty$ such that 
\begin{enumerate}
   \item $\om(2t) = O(\om(t))$  as  $t \to \infty$, \label{om1}
   \item $\log t = o(\om(t))$ as  $t \to \infty$, and \label{om3}
   \item $\vh(t) := \om(e^t)$  is convex.  \label{om4}	
\end{enumerate}
The classes $\cE^{[\om]}$ are defined by 
\[
	\cE^{\{\om\}}(U) := \{f \in \cC^\infty(U) : \A  K \Subset U \E \rh >0 : \|f\|^\om_{K,\rh} <\infty\}
\]
and   
\[
	\cE^{(\om)}(U) := \{f \in \cC^\infty(U) : \A  K \Subset U \A \rh >0 : \|f\|^\om_{K,\rh} <\infty\}
\]
by means of the seminorms
\[
  \|f\|^\om_{K,\rh} := \sup_{x \in K,\,\al \in \N^n} |D^\al f(x)| \exp(-\tfrac{1}{\rh} \vh^*(\rh |\al|)), 
  \quad  f \in \cC^\infty(U),
\]
where $\vh^*(t) := \sup_{s\ge 0} \big(st-\vh(s)\big)$, for $t>0$, is the \emph{Young conjugate} of $\vh$.

These classes were originally introduced by Beurling \cite{Beurling61} and Bj\"orck \cite{Bjoerck66} in terms of 
decay properties of the Fourier transform. The description that we used above is due to Braun, Meise, and Taylor 
\cite{BMT90}.

Braun--Meise--Taylor classes $\cE^{[\om]}$ can be identified as classes $\cE^{[\fM]}$ for suitable weight matrices $\fM$.
In fact, by \cite{RainerSchindl12}, setting 
\[
	W^x_k  := \exp(\tfrac{1}{x} \vh^*(x k)), \quad \text{ for } k \in \N \text{ and } x >0, 
\]
defines a weight matrix $\fW = \{W^x : x >0\}$ 
such that
\[
	\cE^{[\om]} = \cE^{[\fW]} \quad \text{ algebraically and topologically.}
\] 
The weight matrix $\fW$ always satisfies ($B_2$) as well as ($R_2$); cf.\ \cite[(5.6)]{RainerSchindl12}. 
It fulfills ($B_1$) (resp.\ ($R_1$)) if and only if $\om(t) = o(t)$ (resp.\ $\om(t) = O(t)$) as $t \to \infty$; 
cf.\ \cite[Corollary 5.17]{RainerSchindl12}.

Due to \cite{BMM07}, 
$\cE^{[\om]} = \cE^{[M]}$ for some weight sequence $M$ if and only if 
  \begin{align*}
     \E H\ge 1 \A t\ge 0 : 
      2\om(t) \le \om(Ht) + H.   
  \end{align*} 
This is the case if and only if some (equivalently each) $W^x$ has moderate growth, and then   
$\cE^{[\om]} = \cE^{[W^x]}$ for all $x>0$.  

\subsection{Lacunary regularity} \label{sec:lacunary}

In this section we review some results of Liess \cite{Liess:1990aa} which are based on earlier work 
by Baouendi and Metivier \cite{Baouendi:1982aa} and Bolley, Camus, and Metivier \cite{Bolley:1991ab}.
We shall need only one direction of the characterization of Liess (in a special case) and 
try to get by with the minimal assumptions on the weight sequences, but we will otherwise not strive for 
utmost generality. 
Moreover we show the result in the framework of weight matrices.

Let $\fM$ be a weight matrix and let $P$ be a linear partial differential operator with analytic coefficients of order $m$ 
on an open subset $U$ of $\R^n$. Moreover, let $k=(k_j)$ be a strictly increasing sequence of positive integers. 
Then $\cE^{[\fM]}_{P,k}(U)$ denotes the space of all $f \in C^\infty(U)$ such that for all $K \Subset U$ 
there exist $M \in \fM$ and $C,\rh >0$ such that (resp.\ for all $M \in \fM$  and all $\rh >0$ there is $C>0$ such that) 
\[
		\|P^{k_j} f\|_{L^2(K)} \le C \rh^{m k_j} M_{m k_j} \quad \text{ for all } j \in \N. 
\]

\begin{lemma} \label{lem:A1}
	Assume that $\fM$ is a [regular] weight matrix and that
	$k=(k_j)$ satisfies 
\begin{equation} \label{eq:reccurence}
 	a k_j \le k_{j+1} \le a k_j + b \quad \text{ for all } j, 
 \end{equation} 
	for some $a\in \N_{\ge 1}$ and $b \in \N$. 
	Then for $f \in C^\infty(U)$ the following conditions are equivalent:
	\begin{enumerate}
	 	\item $f \in \cE^{[\fM]}(U)$.
	 	\item for all $K \Subset U$ 
there exist $M \in \fM$ and $C,\rh >0$ such that (resp.\ for all $M \in \fM$  and all $\rh >0$ there is $C>0$ such that) 
\[
		\|d_v^{k_j} f\|_{L^2(K)} \le C \rh^{k_j} M_{k_j} \quad \text{ for all } j \in \N \text{ and all } v \in S^{n-1}, 
\] 
	 \end{enumerate} 	
	where $d_v f(x) = \p_t|_{t=0} f(x+tv)$ is the directional derivative.
\end{lemma}

\begin{proof}	
	Only the implication (2) $\Rightarrow$ (1) requires an argument. 
	It follows from the following interpolation formula (see \cite[(5)]{Liess:1990aa}): 
	if $U' \Subset U'' \Subset U$ are open subsets of $\R^n$ and 
	 $\ell \le k$, then for $i = 1,\ldots,n$,  
\begin{equation} \label{eq:interpolation}
	\|D_i^\ell u\|_{L^2(U')} 
	\le C(U',U'')^k  \|u\|_{L^2(U'')}^{1-\ell/k} \Big(k^\ell \|u\|_{L^2(U'')}^{\ell/k} + \|D_i^k u\|_{L^2(U'')}^{\ell/k} \Big).
\end{equation}
	For each fixed $v \in S^{n-1}$ we may choose a basis in which $v$ is the first basis vector.  
	Now it suffices to choose $j$ such that $k_j \le \ell \le k_{j+1}$ and to use \eqref{eq:interpolation} for 
	$k = k_{j+1}$. In the Roumieu case  
	the conditions \eqref{eq:reccurence} and ($R_2$) guarantee that there exist $N,N' \in \fM$ such that 
	\[
		M_{k_{j+1}} \le M_{a k_j +b} \le C^{a k_j +b} N_b N_{a k_j} \le C^{a k_j +b + a(a+1)k_j/2} N'_b (N'_{k_j})^a,
	\]
	and hence using again \eqref{eq:reccurence} we conclude 
	\[
		M_{k_{j+1}}^{1/k_{j+1}} \lesssim (N'_{k_{j}})^{1/k_{j}} \le (N'_{\ell})^{1/\ell}. 	
	\]
	Moreover $k_{j+1}^\ell \le A^\ell k_j^\ell \le A^\ell \ell^\ell \le (A')^\ell M_\ell$ for all $M \in \fM$ and 
	suitable constants $A,A'$, since $(k!/M_k)^{1/k}$ is bounded by ($R_1$).
	In the Beurling case we use ($B$) instead of ($R$) in a similar way. 
	Since $(k!/M_k)^{1/k} \to 0$ in this case, we find that 
	$k_{j+1}^\ell \le  \ep^\ell M_\ell$ for all $M \in \fM$ and all $\ep>0$. 

	In any case we may conclude that (2) actually holds for the sequence $k_j =j$. 
	By the polarization formula \cite[Lemma 7.3(1)]{KM97}, we have 
	\[
		\|d^j f(x)\|_{L_j} \le (2e)^j \sup_{v \in S^{n-1}} |d_v^j f(x)|,
	\]
	where $\|d^j f(x)\|_{L_j}$ denotes the operator norm of the Fr\'echet derivative of order $j$ of $f$ at $x$. 
	Together with the Sobolev inequality and the fact that $\fM$ is stable by derivation (in order to 
	switch from $L^2$- to $L^\infty$-estimates) 
	we find that $f \in \cE^{[\fM]}(U)$. 
\end{proof}

\begin{proposition} \label{lem:A2}
	Let $P$ be some elliptic linear partial differential operator of order $m$ with real analytic coefficients on an open set $U \subseteq \R^n$. 
Let $k=(k_j)$ be a strictly increasing sequence of positive integers satisfying \eqref{eq:reccurence}. 
Assume that $\fM$ is a [regular] weight matrix.
	Then 
	\begin{equation} 
		\cE^{[\fM]}(U) = \cE^{[\fM]}_{P,k}(U).
	\end{equation}
\end{proposition}

\begin{proof}
	The inclusion $\cE^{[\fM]}(U) \subseteq \cE^{[\fM]}_{P,k}(U)$ is clear. 

	For the converse inclusion we follow the arguments of \cite[Proposition 3.3]{Bolley:1991ab}. 
	Let $U' \Subset U'' \Subset U$ be open subsets of $\R^n$. By \cite[Theorem 1.2]{Baouendi:1982aa},
	there is a constant $C>0$ such that for all $k \in \N$ there is $\ch_k \in C^\infty_c(U'')$ with $\ch_k = 1$ on $U'$ 
	and with values in $[0,1]$ such that for all $u \in C^\infty(\ol {U''})$
	\begin{equation*}
	 	|\xi|^{m k} |\widehat{(\ch_k u)}(\xi)| \le C^{k+1} \Big( \|P^k u\|_{L^2(U'')} + k!^m \|u\|_{L^2(U'')}\Big), \quad \xi \in \R^n, k \in \N. 
	 \end{equation*} 
	Thus if $u \in \cE^{[\fM]}_{P,k}(U)$ then there exist $M \in \fM$ and $C,\rh>0$  
	(resp.\ for all $M \in \fM$ and all $\rh>0$ there is $C$) such that
	\[
		|\widehat{(\ch_{k_j} u)}(\xi)| \le C  \rh^{m k_j}|\xi|^{-m k_j} M_{m k_j}.
	\]
	Fix $v \in S^{n-1}$.  
	In view of 
	\begin{align*}
		(-i d_v)^{mk_j - n-1} u(x) = 
		\frac{1}{(2\pi)^n}  \int e^{ i \langle x ,\xi \rangle} \langle v,\xi\rangle^{mk_j - n-1} \widehat{\ch_{k_j} u}(\xi)\, d\xi
	\end{align*}
	we conclude that, for all $j$,
	\begin{align*}
		\|d_v^{mk_j - n-1} u\|_{L^\infty(U')} \le C \rh^{mk_j} M_{m k_j} \le \tilde C {\tilde \rh}^{mk_j} N_{m k_j-n-1},
	\end{align*}
	where we used ($R_2$) in the Roumieu case. In the Beurling case we use ($B_2$) instead.
	The constants $\tilde C, \tilde \rh$ are independent of $v \in S^{n-1}$. 
	It remains to  apply \Cref{lem:A1} for the sequence $j \mapsto mk_j - n-1$.
\end{proof}

\begin{remark}
	It is possible to just work with $\p_1,\ldots,\p_n$ instead of all directional derivatives $d_v$. 
	Then the equivalence of (1) and (2) in \Cref{lem:A1} would read:
	\[
		\cE^{[\fM]}(U) = \bigcap_{i=1}^n \cE^{[\fM]}_{\p_i,k}(U).
	\]
	The proof is similar, but in the end one has to show the stronger statement
	\[
		\bigcap_{i=1}^n \cE^{[\fM]}_{\p_i,k}(U) = \bigcap_{i=1}^n \cE^{[\fM]}_{\p_i,(j)_j}(U).
	\] 
\end{remark}

\subsection{Gelfand--Shilov classes}

Let $\fM$ be a weight matrix. We consider the \emph{Gelfand--Shilov classes} $\cS^{[\fM]}(\R^n)$ consisting of all 
$f \in \cS(\R^n)$ such that there exist $M \in \fM$ and $\rh>0$ with (resp.\ for all $M \in \fM$ and all $\rh>0$)
\[ \sup_{k,\ell \in \N}    
\sup_{|\al| = \ell} \frac{\| |x|^k  D^\al f \|_{L^\infty(\R^n)}}{\rh^{k + \ell} M_k M_\ell} < \infty. 
\]
The next lemma is a slight generalization of \cite[Theorem 2.3]{Chung:1996aa}.

\begin{lemma} \label{lem:A3}
	Assume that $\fM$ is a [regular] weight matrix.
	Then the following are equivalent.
	\begin{enumerate}
		\item $f \in \cS^{[\fM]}(\R^n)$.
		\item There exist $M\in \fM$ and constants $C,\rh,\si >0$ (for all $M \in \fM$ and for all $\rh,\si >$ there is $C>0$) such that 
		\[
			\sup_x |x^\al f(x)| \le C \rh^{|\al|} M_{|\al|} \quad \text{ and } \quad \sup_x |D^\be f(x)| \le C \si^{|\be|} M_{|\be|} 
		\]  
		for all $\al,\be \in \N^n$.
		\item There exist $M\in \fM$ and constants $C,\rh,\si >0$ (for all $M \in \fM$ and for all $\rh,\si >$ there is $C>0$) such that 
		\[
			\sup_x |x^\al f(x)| \le C \rh^{|\al|} M_{|\al|} \quad \text{ and } \quad \sup_{\xi} |\xi^\be \widehat f(\xi)| \le C \si^{|\be|} M_{|\be|} 
		\]  
		for all $\al,\be \in \N^n$. 
	\end{enumerate}
\end{lemma}

\begin{proof}
	It is straightforward to adapt the proof of \cite[Theorem 2.3]{Chung:1996aa}.
\end{proof}

An obvious consequence of the lemma is that $\cS^{[\fM]}(\R^n)$ is invariant under the Fourier transform:
$f \in \cS^{[\fM]}(\R^n)$ if and only if $\widehat f \in \cS^{[\fM]}(\R^n)$.

\section{Ultradifferentiable Chevalley theorems} \label{sec:Chevalley}

\subsection{The setting}

The following facts can be found in \cite{Helgason:1962aa}; see also \cite{Dadok:1982aa} whose arguments we follow closely.

Let $\fg$ be a real semisimple Lie algebra of noncompact type and let $\fg = \fk \oplus \fp$ be 
a Cartan decomposition. 
Then $\fk$ is a subalgebra of $\fg$ which is the Lie algebra of the maximal compact subgroup $K$ of 
the adjoint group $\on{Int}(\fg)$.  
The decomposition is direct with respect to the Killing form. 
Let $\fa \subseteq \fp$ be a maximal abelian subspace and 
\[
	\fg = \fg^0 \oplus \bigoplus_{\al \in \Si} \fg^\al 
\]
the root space decomposition with respect to $\fa$, where 
$$\fg^\al = \{ X \in \fg : (\on{ad} H) X = \al(H) X \text{ for all } H \in \fa\}$$ 
and $\Si$ is the set of roots $0 \ne \al \in \fa^*$ with $\fg^\al \ne 0$.
We set $m_\al := \dim \fg^\al$.

Choose a Weyl chamber $\fa^+ \subseteq \fa$, i.e., a connected component of the complement 
of the union of hyperplanes in $\fa$ defined by the roots $\al \in \Si$.
Let $\Si^+$ denote the collection of positive roots w.r.t.\ $\fa^+$. 
The adjoint action of $K$ on $\fp$ preserves the inner product induced by the Killing form. 
Every $K$-orbit intersects $\fa$ orthogonally in an orbit of the Weyl group. 
The Weyl group 
\[
	W = N_K(\fa)/Z_K(\fa),  
\]
where 
$$N_K(\fa) = \{ k \in K : \on{Ad}(k) \fa = \fa\}$$ 
and 
$$Z_K(\fa) = \{ k \in K : \on{Ad}(k) H = H \text{ for all } H \in \fa\},$$
is a finite group of linear automorphisms of $\fa$ which is generated by the 
reflections in the hyperplanes $\{H \in \fa: \al(H) = 0\}$, for $\al \in \Si^+$.

If $M$ denotes the centralizer of $\fa$ in $K$, then 
$$K/M \times \fa^+ \ni (kM,H) \mapsto \on{Ad}(k)H$$
is a diffeomorphism onto an open and dense subset of $\fp$. There is the following integral formula for 
$f \in C^\infty_c(\fp)$,
\begin{equation} \label{change}
	\int_\fp f(x) \, dx = \int_{\fa^+} \int_{K/M} f(\on{Ad}(k)H) \prod_{\al \in \Si^+} \al(H)^{m_\al} \, dk \, dH,
\end{equation}
where $dx$ and $dH$ are the Lebesgue measures on $\fp$ and $\fa$, respectively, and $dk$ is an invariant measure on $K/M$, 
all of them with suitable normalizations; see \cite[p.380]{Helgason:1962aa}.

We denote by $C^\infty(\fp)^K$ the space of $C^\infty$-functions on $\fp$ which are invariant under 
the adjoint	action of $K$ on $\fp$. Similarly $C^\infty(\fa)^W$ is the space of $W$-invariant $C^\infty$-functions on $\fa$.
By $\De_{\fp}$ we mean the flat Euclidean Laplace operator on $\fp$.
For $f \in C^\infty(\fp)^K$ we have  
\[
	(\De_{\fp} f)|_{\fa^+} = \on{rad}(\De_{\fp}) (f|_{\fa^+}),
\]
where $\on{rad}(\De_{\fp})$ is a differential operator on $\fa^+$ called the \emph{radial part} of $\De_{\fp}$. 
Then (see \cite[Proposition 1.1]{Dadok:1982aa}) 
\begin{equation} \label{rad}
	\on{rad}(\De_{\fp})  = \De_{\fa} + \sum_{\al \in \Si^+} m_\al \frac{\on{grad} \al}{\al},	
\end{equation}
where $\De_{\fa}$ is the Laplace operator on $\fa$ and $(\on{grad} \al)(f) := \langle \on{grad} \al , \on{grad} f\rangle$. 

\subsection{Chevalley's theorem in local ultradifferentiable classes}

\begin{theorem} \label{Chevalley}
	Let $\fM$ be a [regular] weight matrix. 
	Then the restriction mapping  $\cE^{[\fM]}(\fp)^K \to \cE^{[\fM]}(\fa)^W$ is an isomorphism. 
\end{theorem}

\begin{proof}
	Every $f \in C^\infty(\fa)^W$ 
	can be extended to a continuous function $\widetilde f$ on $\fp$, by making it constant on the $K$-orbits.
	The continuity of $\widetilde f$ follows from \cite[Proposition 2.4]{Helgason:1980aa}; in fact, for $X,Y \in \fp$ 
	one has 
	\[
		\on{dist}(\on{Ad}(K) X,\on{Ad}(K) Y) = \on{dist}(\on{Ad}(K) X \cap \fa^+,\on{Ad}(K) Y \cap \fa^+). 
	\] 
	Using \eqref{change} and \eqref{rad}, it is not hard to see that
	\begin{equation} \label{key}
		\De_{\fp} \widetilde f = (\on{rad}(\De_{\fp}) f){\,\widetilde{~~~}}
	\end{equation}
	in the sense of distributions; for details see \cite[pp.124-125]{Dadok:1982aa}.
	 	
	Since $f$ is invariant with respect to the reflection through the hyperplane $\{\al = 0\}$, 
	the function $(\on{grad} \al)(f)$ vanishes on $\{\al = 0\}$ whence $(\on{grad} \al)(f)/\al$ is smooth. 
	Then, by \eqref{rad}, we may conclude that  
	$\on{rad}(\De_{\fp}) f \in C^\infty(\fa)^W$, where now $\on{rad}(\De_{\fp})$ is the obvious extension to $\fa$ 
	as a $W$-invariant differential operator (with singularities along the hyperplanes $\{\al = 0\}$, for $\al \in \Si^+$).

	Iterating \eqref{key} yields 
	\begin{equation} \label{key2}
		   \De_{\fp}^m \widetilde f = (\on{rad}(\De_{\fp})^m f){\,\widetilde{~~~}}, \quad m \ge 1.  	
	\end{equation}    
	This implies that $\widetilde f \in C^\infty(\fp)^K$ by means of elliptic regularity; cf.\ \cite[p.124]{Dadok:1982aa}.

	Now suppose that $f \in \cE^{[\fM]}(\fa)^W$. By the above, the $K$-invariant extension $\widetilde f$ is smooth.
	Let $U$ be a relatively compact open subset of $\fp$ and let $V$ be its saturation with respect to the $K$-action, 
	i.e., the union of all $K$-orbits that meet $U$. 
	By \eqref{key2},  
	\begin{equation} \label{relation}
		 \|\De_{\fp}^m \widetilde f\|_{L^\infty(V)}	= \| (\on{rad}(\De_{\fp})^m f){\,\widetilde{~~~}}\, \|_{L^\infty(V)} = 
		 \| \on{rad}(\De_{\fp})^m f \|_{L^\infty(V \cap \fa)}.
	\end{equation}  
	We claim that there exist $M \in \fM$ and constants $C,\rh >0$ 
	(resp.\ for all $M \in \fM$ and all $\rh >0$ there is $C>0$) such that  
	\begin{equation} \label{estimate}
		\| \on{rad}(\De_{\fp})^m f \|_{L^\infty(V \cap \fa)} \le C \rh^{2m} M_{2m}, \quad \text{ for all } m \ge 1.
	\end{equation}
	Then, by \eqref{relation}, we may conclude
	\begin{equation*}
		\|\De_{\fp}^m \widetilde f\|_{L^\infty(V)}  \le C \rh^{2m} M_{2m}, \quad \text{ for all } m \ge 1.	
	\end{equation*}
	That implies that $\widetilde f \in \cE^{[\fM]}_{\De_{\fp},(j)_j}(\fp)$. 
	By \Cref{lem:A2}, we find that  $\widetilde f \in \cE^{[\fM]}(\fp)^K$.

	It remains to show the claim \eqref{estimate}. 
	We use that, for any $f \in C^\infty(\fa)^W$, the 
	function $$F:= \frac{(\on{grad} \al)(f)}{|\on{grad} \al|}$$ vanishes on the hyperplane $\{\al =0\}$. 
	Although $\on{rad}(\De_{\fp})$ is a differential operator with singularities 
	along the hyperplanes $\{\al = 0\}$, for $\al \in \Si^+$, its action on $W$-invariant functions $f$ is well-behaved. 
	Indeed, 
	if we set 
	\[
			v_\al = \frac{\on{grad} \al}{|\on{grad} \al|}
	\]
	and denote by 
	\[
		y_\al(x) := x -  \frac{\al(x)}{|\on{grad} \al|} v_\al,
	\]
	the orthogonal projection on the hyperplane $\{\al = 0\}$, 
	then since $F$ vanishes on $\{\al = 0\}$ we have for all $x \in \fa$,
	\begin{align*}
		F(x) &= \int_0^1 \frac{d}{dt} F\big(t \tfrac{\al(x)}{|\on{grad} \al|}  v_\al + y_\al(x)\big) \,dt
			\\
			&= \int_0^1 dF\big(t \tfrac{\al(x)}{|\on{grad} \al|}  v_\al + y_\al(x)\big) (\tfrac{\al(x)}{|\on{grad} \al|}  v_\al) \,dt
			\\
			&=\frac{\al(x)}{|\on{grad} \al|} \int_0^1 d_{v_\al} F\big(t \tfrac{\al(x)}{|\on{grad} \al|}  v_\al + y_\al(x)\big) \,dt. 		
	\end{align*}
	Since $F =\frac{(\on{grad} \al)(f)}{|\on{grad} \al|} = d_{v_\al} f$, we obtain 
	\begin{align*}
		\Big(\frac{\grad \al}{\al}\Big) (f)(x) 
		&=  \int_0^1 d_{v_\al} \big(d_{v_\al} f\big)\big(t \tfrac{\al(x)}{|\on{grad} \al|}  v_\al + y_\al(x)\big) \,dt
		\\
		&=  \int_0^1 d_{ v_\al}^2 f \big(t \tfrac{\al(x)}{|\on{grad} \al|}  v_\al + y_\al(x)\big) \,dt.
	\end{align*}
	By \eqref{rad}, we see that for $f \in C^\infty(\fa)^W$ we have
	\begin{equation*} 
	\on{rad}(\De_{\fp})f(x)  = (\De_{\fa} f)(x) + \sum_{\al \in \Si^+} m_\al  \int_0^1 d_{ v_\al}^2 f \big(t \tfrac{\al(x)}{|\on{grad} \al|}  v_\al + y_\al(x)\big) \,dt.	
	\end{equation*}
	Since $\on{rad}(\De_{\fp})f \in C^\infty(\fa)^W$ we can replace $f$ in the above formula by 
	$\on{rad}(\De_{\fp})f$ and iterate this procedure in order to express 
	$\on{rad}(\De_{\fp})^m(f)$ for $m \ge 1$ in terms of combinations of powers of differential operators $\De_\fa$ and 
	$d_{v_\al}$ for $\al \in \Si^+$ applied to $f$. 
	Using the linearity of the operators and the linearity of $\al$ and $y_\al$ and 
	computing $L^\infty$-norms on balls centered at the origin in $\fa$,  
	it is then straightforward to conclude \eqref{estimate}.	
\end{proof}

\begin{corollary} \label{cor1}
	Let $M$ be a weight sequence with moderate growth. Then 
	the restriction mapping  $\cE^{[M]}(\fp)^K \to \cE^{[M]}(\fa)^W$ is an isomorphism, 
	if we assume $(k!/M_k)^{1/k}$ is bounded in the Roumieu case and $(k!/M_k)^{1/k} \to 0$ as $k \to \infty$ in the Beurling case.
\end{corollary}

\begin{corollary} \label{cor2}
	Let $\om$ be a weight function. Then 
	the restriction mapping  $\cE^{[\om]}(\fp)^K \to \cE^{[\om]}(\fa)^W$ is an isomorphism, 
	if we assume $\om(t) = O(t)$ as $t \to \infty$ in the Roumieu case and $\om(t) = o(t)$  as $t \to \infty$ in the Beurling case. 
\end{corollary}

\begin{corollary} \label{cor3}
	The restriction mapping  $C^\om(\fp)^K \to C^\om(\fa)^W$ is an isomorphism.
\end{corollary}

\subsection{Chevalley's theorem in Gelfand--Shilov classes}

As a consequence of the $C^\infty$ Chevalley theorem Dadok \cite[Corollary 1.5]{Dadok:1982aa} showed that every $W$-invariant  
Schwartz function $f \in \cS(\fa)^W$ extends to a Schwartz function $\widetilde f \in \cS(\fp)^K$; see 
also Helgason \cite[Proposition 2.3]{Helgason:1980aa} for a different proof.

\begin{theorem} \label{GelfandShilov}
	Let $\fM$ be a [regular] weight matrix. 
	Then the restriction mapping  $\cS^{[\fM]}(\fp)^K \to \cS^{[\fM]}(\fa)^W$ is an isomorphism.
\end{theorem}

\begin{proof}
Let $f  \in \cS^{[\fM]}(\fa)^W$. 
We want to show that the $K$-invariant extension $\widetilde f$ of $f$ to $\fp$ is of class $\cS^{[\fM]}$.
We already know that $\widetilde f$ is smooth. 
Choose linear coordinates in $\fp$ such that the $K$-invariant inner product induced by the Killing form is given by 
$\langle x,y \rangle = x_1 y_1 + \cdots + x_n y_n$. 
Then $|x|^2 = \langle x, x \rangle$ is a $K$-invariant polynomial. By \Cref{lem:A3}, it suffices to check that 
there exist $M \in \fM$ and constants $C,\rh,\si>0$ 
(resp.\ for all $M \in \fM$ and all $\rh,\si>0$ there is $C>0$)
such that 
\begin{equation} \label{GSestimate}
	\sup_x |x|^{2k} |\widetilde f(x)| \le C \rh^{2k} M_{2k}  \quad \text{ and } 
	\quad \sup_\xi |\xi|^{2k} |(\widetilde f)^\wedge(\xi)| \le C \si^{2k} M_{2k}
\end{equation}
for all $k$.
Here $(\widetilde f)^\wedge$ denotes the Fourier transform of $\widetilde f$.  
(We use that for any $\al \in \N^n$ we have 
$|x^\al| \le |x|^{|\al|} \le |x|^{2k} \max\{1,|x|^2\}$, where $k = \lfloor |\al|/2 \rfloor$, and that $\fM$ is 
[regular].)

The first estimate in \eqref{GSestimate} simply follows from the assumption $f  \in \cS^{[\fM]}(\fa)^W$ and 
the fact that 
$$\big\| | \cdot |^{2k}  \widetilde f(\cdot) \big\|_{L^\infty(\fp)} 
= \big\| | \cdot |^{2k} f(\cdot)\big\|_{L^\infty(\fa)}.$$ 
For the second estimate in \eqref{GSestimate} we observe that for all $\xi \in \fp$ (where $n = \dim \fp$)
\begin{align*}
 |\xi|^{2k} | (\widetilde f)^\wedge(\xi)| \le |(\De_\fp^k \widetilde f)^\wedge(\xi)| 
&\le \int_\fp (1+ |x|^{2n})|\De_\fp^k \widetilde f(x)| \, \frac{dx}{(1+ |x|^{2n})} 	
\\
&\le \big\|(1+ |\cdot|^{2n}) \De_\fp^k  \widetilde f(\cdot)\big\|_{L^\infty(\fp)} \int_\fp  \, \frac{dx}{(1+ |x|^{2n})}
\\
&\le  C(n)\, \big\|(1+ |\cdot |^{2n}) \on{rad}(\De_\fp)^k f(\cdot)\big\|_{L^\infty(\fa)},
\end{align*}
where in the last inequality we used \eqref{key2}.
The assumption $f  \in \cS^{[\fM]}(\fa)^W$ together with the fact that $\fM$ is [regular] and the justification for \eqref{estimate} 
yields the required estimate.
\end{proof}

\section{Isotropic functions}

	Let $\fg =\mathfrak{sl}(n,\R)$ be the Lie algebra of $n \times n$ real matrices with trace zero 
	and consider the Cartan decomposition $\fg = \fk \oplus \fp$, where $\fk= \mathfrak{so}(n,\R)$ is the Lie algebra 
	of skew-symmetric matrices and $\fp = \on{Sym}(n)_0$ are the symmetric matrices with trace zero. 
	Then $K = \on{SO}(n)$ acts by conjugation on $\fp$  
	and the maximal subalgebra $\fa = \on{Diag}(n)_0$ consists of the diagonal matrices with trace zero. 
	The Weyl group is isomorphic to the symmetric group $\on {S}_n$ and acts on $\fa$ be permuting the diagonal entries. 

	Let $F : \on{Diag}(n) \cong \R^n \to \R$ be a symmetric function. 
	We extend $F$ to an isotropic function $f : \on{Sym}(n) \to \R$ by setting $f(A) = F(a_1,\ldots,a_n)$, where 
	$a_1 \ge a_2 \ge \cdots \ge a_n$ are the eigenvalues of $A$. 
	Setting $B= A - \tfrac{1}{n}(\Tr A) \mathbb I$ and $b_j := a_j - \tfrac{1}{n}\sum_{i=1}^n a_i$, for $j=1,\ldots,n$, we have
	\begin{align*}
		g(B) &:= f(B + \tfrac{1}{n}(\Tr A) \mathbb I) 
		\\
		&= f(A) 
		\\
		&= F(a_1,\ldots,a_n) 
		\\
		&=  F\Big(b_1 + \tfrac{1}{n}\sum_{i=1}^n a_i,\ldots,b_n +\tfrac{1}{n}\sum_{i=1}^n a_i\Big) 
		\\
		&=: G(b_1,\ldots,b_n),
	\end{align*}
	where $g : \on{Sym}(n)_0 \to \R$ is isotropic and $G : \on{Diag}(n)_0 \to \R$ is symmetric. 
	Moreover, if $F$ is $C^\infty$, of class $\cE^{[\fM]}$, or of class $\cS^{[\fM]}$ 
	(for a [regular] weight matrix), then $G$ is of the same class, 
	since $F$ is the composite of $G$ with the orthogonal projection onto $\on{Diag}(n)_0$.   
	Then \Cref{Chevalley} and \Cref{GelfandShilov} imply that 
	$g$, and thus $f$, is of the corresponding class. So we obtain

	\begin{theorem} \label{isotropic}
		Let $\fM$ be a [regular] weight matrix.
		An isotropic function $f : \on{Sym}(n) \to \R$ is of class $\cE^{[\fM]}$ 
		(resp.\ $\cS^{[\fM]}$) if and only if its symmetric companion $F : \R^n \to \R$ is of class $\cE^{[\fM]}$ 
		(resp.\ $\cS^{[\fM]}$).  
	\end{theorem}

	Clearly, we immediately get isotropic versions of \Cref{cor1,cor2,cor3}.

	\begin{remark}
		Alternatively, it is possible to use the inductive formula for the derivatives of $f$ derived in \cite{Silhavy2000} 
		in order to give a direct proof of \Cref{isotropic}.
	\end{remark}


\def\cprime{$'$}
\providecommand{\bysame}{\leavevmode\hbox to3em{\hrulefill}\thinspace}
\providecommand{\MR}{\relax\ifhmode\unskip\space\fi MR }
\providecommand{\MRhref}[2]{%
  \href{http://www.ams.org/mathscinet-getitem?mr=#1}{#2}
}
\providecommand{\href}[2]{#2}

\end{document}